\documentclass[MSNbibl,number,citesort,seceqn,dvips]{arxbj}
\usepackage{mathrsfs}

\aid{0}
\volume{19}
\issue{5A}
\pubyear{2013}
\firstpage{1855}
\lastpage{1879}
\doi{10.3150/12-BEJ433} 

\makeatletter
\newtheorem{theo}{Theorem}[section]
\newremark{remark}[theo]{Remark}
\newtheorem{corol}[theo]{Corollary}
\newtheorem{lemme}[theo]{Lemma}

\newcommand{\dE}{\mathbb{E}}
\newcommand{\dN}{\mathbb{N}}
\newcommand{\dP}{\mathbb{P}}
\newcommand{\dR}{\mathbb{R}}
\newcommand{\rP}{\mathrm{P}}
\newcommand{\cA}{\mathcal{A}}
\newcommand{\cD}{\mathcal{D}}
\newcommand{\cE}{\mathcal{E}}
\newcommand{\cF}{\mathcal{F}}
\newcommand{\cI}{\mathcal{I}}
\newcommand{\cL}{\mathcal{L}}
\newcommand{\cP}{\mathcal{P}}
\newcommand{\cW}{\mathcal{W}}
\newcommand{\sB}{\mathscr{B}}
\newcommand{\sL}{\mathscr{L}}
\newcommand{\Ga}{\Gamma}
\newcommand{\la}{\lambda}
\newcommand{\si}{\sigma}
\newcommand{\vphi}{\varphi}
\newcommand{\pd}{{\partial}} 
\newcommand{\cleq}{\leq_{\mathrm{c}}}
\newcommand{\Var}{\operatorname{Var}}
\newcommand{\Ent}{\operatorname{Ent}}
\newcommand{\Cov}{\operatorname{Cov}}
\newcommand{\Lip}{\operatorname{Lip}}
\newcommand{\Marg}{\operatorname{Marg}}

\newcommand{\eqref}[1]{(\ref{#1})}

\makeatother

\begin{document}
\begin{frontmatter}

\title{Intertwining and commutation relations for birth--death processes}
\runtitle{Intertwining and commutation for birth--death processes}

\begin{aug}
\author[1]{\fnms{Djalil} \snm{Chafa\"{i}}\corref{}\thanksref{1}\ead[label=e1]{djalil@chafai.net}\ead[label=u1,url]{http://djalil.chafai.net/}} \and
\author[2]{\fnms{Ald\'eric} \snm{Joulin}\thanksref{2}\ead[label=e2]{ajoulin@insa-toulouse.fr}\ead[label=u2,url]{http://www-gmm.insa-toulouse.fr/\textasciitilde ajoulin/}}
\runauthor{D. Chafa\"{i} and A. Joulin}
\address[1]{Laboratoire d'Analyse et de Math\'ematiques Appliqu\'ees, CNRS UMR8050, Universit\'e Paris-Est Marne-la-Vall\'ee,
France. \printead{e1,u1}}
\address[2]{Institut National des Sciences
Appliqu\'ees de Toulouse,
Institut de Math\'ematiques de Toulouse, Universit\'e de Toulouse, CNRS UMR5219, France. \printead{e2};\\ \printead{u2}}
\end{aug}

\received{\smonth{8} \syear{2011}}
\revised{\smonth{3} \syear{2012}}

%
\begin{abstract}
Given a birth--death process on $\dN$ with semigroup $(P_t)_{t\geq0}$
and a
discrete gradient $\pd_u$ depending on a positive weight $u$, we establish
intertwining relations of the form $\pd_u P_t = Q_t\,\pd_u $, where
$(Q_t)_{t\geq0}$ is the Feynman--Kac semigroup with potential $V_u$ of
another birth--death process. We provide applications when $V_u$ is
nonnegative and uniformly bounded from below, including Lipschitz
contraction and Wasserstein curvature, various functional inequalities, and
stochastic orderings. Our analysis is naturally connected to the previous
works of Caputo--Dai Pra--Posta and of Chen on birth--death processes. The
proofs are remarkably simple and rely on interpolation, commutation, and
convexity.
\end{abstract}

%
\end{frontmatter}

\section{Introduction}\label{sectintro}

Commutation relations and convexity are useful tools for the fine
analysis of
Markov diffusion semigroups \cite{bakemery,bakry2,ledoux}. The
situation is
more delicate on discrete spaces, due to the lack of a chain rule formula
\cite{bobkovledoux,ane,chen1,jp,bobkovtetali,chafai2,caputo,chen3}. In this
work, we obtain new intertwining and sub-commutation relations for a
class of
birth--death processes involving a discrete gradient and an auxiliary
Feynman--Kac semigroup. We also provide various applications of these
relations. Our analysis is naturally related to the curvature condition of
Caputo--Dai Pra--Posta \cite{caputo} and to the Chen exponent of Chen
\cite{chen1,chen3}. More precisely, let us consider a birth--death process
$(X_t)_{t\geq0}$ on the state space $\dN:= \{ 0,1,2, \ldots\}$, that is, a
Markov process with transition probabilities given by
\[
P_t^x (y) = \dP_x (X_t =y) = %
\cases{
\la_x t + \mathrm{o}(t)  & \quad if $y=x+1$, \cr
\nu_x t + \mathrm{o}(t) & \quad  if $y=x-1$, \cr
1- (\la_x + \nu_x) t + \mathrm{o}(t) & \quad  if $y=x$,}
\]
where $\lim_{t\to0}t^{-1}\mathrm{o}(t)=0$. The transition rates $\la$ and
$\nu$ are
respectively called the birth and death rates of the process
$(X_t)_{t\geq
0}$. The process is irreducible, positive recurrent (or ergodic), and
nonexplosive when the rates satisfy to $\la>0$ on $\dN$ and
$\nu>0$ on $\dN^*$ and $\nu_0 = 0$ and
\[
\sum_{x=1}^\infty\frac{\la_0 \la_1 \cdots\la_{x-1}}{\nu_1 \nu
_2 \cdots\nu_x} <\infty \quad
\mbox{and} \quad
\sum_{x=1}^\infty%
{{\biggl(\frac{1}{\la_x}+\frac{\nu_x}{\la_x\la_{x-1}}+\cdots+\frac
{\nu_x\cdots\nu_1}{\la_x\cdots\la_1\la_0}\biggr)}} =
\infty,
\]
respectively. In this case, the unique stationary distribution $\mu$
of the
process is reversible and is given by
%
\begin{equation}\label{eqinvariant}
\mu(x) = \mu(0) \prod_{y=1}^x \frac{\la_{y-1}}{\nu_y} ,\qquad x\in\dN %
\mbox{ with } %
\mu(0) := {{\Biggl(1+\sum_{x=1}^\infty\frac{\la_0\la_1\cdots\la
_{x-1}}{\nu_1\nu_2\cdots\nu_x}\Biggr)}}^{-1} .
\end{equation}
Let us denote by $\cF$ (resp., $\cF_{+}$ and $\cF_{d}$) the
space of
real-valued (resp., positive and nonnegative nondecreasing) functions
$f$ on $\dN$. The associated semigroup $(P_t )_{t\geq0}$ is defined
for any
bounded or nonnegative function $f$ as
\[
P_t f (x) = \dE_x [f(X_t)] = \sum_{y=0}^\infty f(y) P_t^x (y) ,\qquad x\in
\dN.
\]
This family of operators is positivity preserving and contractive on $L^p
(\mu)$, $p\in[1,\infty]$. Moreover, the semigroup is also symmetric in
$L^2(\mu)$ since $\la_x\mu(x) = \nu_{1+x}\mu(1+x)$ for any $x\in
\dN$ (detailed
balance equation). The generator $\cL$ of the process is given for any
$f\in
\cF$ and $x\in\dN$ by
\begin{eqnarray*}
\cL f(x)
& =& \la_x \bigl( f(x+1) -f(x) \bigr) + \nu_x \bigl( f(x-1) -f(x) \bigr) \\
& =& \la_x \,\pd f (x) + \nu_x \,\pd^* f(x),
\end{eqnarray*}
where $\pd$ and $\pd^*$ are, respectively, the forward and backward discrete
gradients on $\dN$:
\[
\pd f(x) %
:= f(x+1)-f(x)\quad \mbox{and}\quad \pd^* f(x) := f(x-1)-f(x) .
\]
Our approach is inspired from the remarkable properties of two special
birth--death processes: the $M/M/1$ and the $M/M/\infty$ queues. The
$M/M/\infty$ queue has rates $\la_x=\la$ and $\nu_x=\nu x$ for positive
constants $\la$ and $\nu$. It is positive recurrent and its stationary
distribution is the Poisson measure $\mu_\rho$ with mean $\rho=\la
/\mu$.
If $\sB_{x,p}$ stands for the binomial distribution of size $x\in\dN
$ and
parameter $p \in[0,1]$, the $M/M/\infty$ process satisfies for every
$x\in\dN$ and $t\geq0$ to the Mehler type formula
%
\begin{equation}\label{eqmehler}
\sL(X_t |X_0 = x) = \sB_{x, \mathrm{e}^{-\nu t}} \ast\mu_{\rho(1-\mathrm{e}^{-\nu t})}.
\end{equation}
The $M/M/1$ queening process has rates $\la_x=\la$ and $\nu_x=\nu
\mathbf{1}_{\dN\setminus\{0\}}$ where $0<\la<\nu$ are constants.
It is a
positive recurrent random walk on $\dN$ reflected at $0$. Its stationary
distribution $\mu$ is the geometric measure with parameter $\rho:=
\la/\nu$
given by $\mu(x) = (1-\rho)\rho^x$ for all $x\in\dN$. A remarkable common
property shared by the $M/M/1$ and $M/M/\infty$ processes is the intertwining
relation
%
\begin{equation}\label{eqLegalite}
\pd\cL= \cL^{V}\, \pd,
\end{equation}
where $\cL^{V}=\cL-V$ is the discrete Schr\"odinger operator with potential
$V$ given by
\begin{itemize}
\item$V(x) := \nu$ in the case of the $M/M/\infty$ queue
\item$V(x) := \nu\mathbf{1}_{\{0\}}(x)$ for the $M/M/1$ queue.
\end{itemize}
Since $V\geq0$ in these two cases, the operator $\cL^{V}$ is the generator
of a birth--death process with killing rate $V$ and the associated Feynman--Kac
semigroup $(P_t^{V})_{t\geq0}$ is given by
\[
P_t^{V} f(x) %
= \dE_x \biggl[ f(X_t) \exp\biggl(-\int_0^t V(X_s)\, \mathrm{d}s \biggr) \biggr].
\]
The intertwining relation \eqref{eqLegalite} is the infinitesimal
version at
time $t=0$ of the semigroup intertwining
%
\begin{equation}\label{eqegalite}
\pd P_t f (x) = P_t^{V}\, \pd f (x) %
= \dE_x \biggl[ \pd f(X_t) \exp\biggl( - \int_0^t V(X_s)\, \mathrm{d}s
\biggr) \biggr] .
\end{equation}
Conversely, one may deduce
\eqref{eqegalite} from \eqref{eqLegalite} by using a semigroup
interpolation. Namely, if we consider $s\in[0,t]\mapsto J(s) :=
P_s^{V}\, \pd
P_{t-s} f$ with $V$ as above, then \eqref{eqegalite} rewrites as $J(0) =
J(t)$ and \eqref{eqegalite} follows from \eqref{eqLegalite} since
\[
J'(s) = P_s^{V} ( \cL^{V}\, \pd P_{t-s} f - \pd\cL P_{t-s} f ) =0.
\]

In Section \ref{sectcommut}, we obtain by using semigroup
interpolation an
intertwining relation similar to \eqref{eqegalite} for more general
birth--death processes. By using convexity as an additional ingredient,
we also
obtain sub-commutation relations. These results are new and have
several applications
explored in Section \ref{sectappli}, including Lipschitz contraction and
Wasserstein curvature (Section \ref{sectwass}), functional inequalities
including Poincar\'e, entropic, isoperimetric and
transportation-information inequalities (Section
\ref{sectFI}), hitting time of the origin for the $M/M/1$ queue (Section
\ref{secthit}), convex domination and stochastic orderings (Section~\ref{sectconvdom}).

\section{Intertwining relations and sub-commutations}\label{sectcommut}

Let us fix some $u \in\cF_{+}$. The $u$-modification of the original
process $(X_t)_{t\geq0}$ is a birth--death process $(X_{u, t})_{t\geq
0}$ with
semigroup $(P_{u,t})_{t\geq0}$ and generator $\cL_u$ given by
\[
\cL_u f(x) = \la^u_x \,\pd f (x) + \nu^u_x \,\pd^* f(x),
\]
where the birth and death rates are respectively given by
\[
\la^u_x := \frac{u_{x+1}}{u_x} \la_{x+1}
\quad\mbox{and}\quad
\nu^u_x := \frac{u_{x-1}}{u_x} \nu_x .
\]
One can check that the measure $\la u^2\mu$ is symmetric for
$(X_{u,t})_{t\geq0}$. As consequence, the process $(X_{u,t})_{t\geq
0}$ is
positive recurrent if and only if $\la u^2$ is $\mu$-integrable. From
now on, we restrict to the minimal solution corresponding to the
forward and
backward Kolmogorov equations given as follows: for any function $f\in
\cF$ with finite support and $t\geq
0$,
\[
\frac{\mathrm{d}}{\mathrm{d}t} P_{u,t} f%
= P_{u,t} \cL_u f %
= \cL_u P_{u,t} f,
\]
cf. \cite{chen2}, Theorem 2.21. In order to justify in all
circumstances the computations present in these notes, we need to
extend these identities to bounded functions $f$. Although it is not
restrictive for the backward equation, the forward equation is more
subtle and requires an additional integrability assumption. From now
on, we always assume that the transition rates $\lambda^u$ and $\nu
^u$ and also the potential $V_u$ are $P_{u,t}$ integrable.

We define the discrete gradient $\pd_u$ and the
potential $V_u$ by
\[
\pd_u := (1/u)\,\pd%
\quad\mbox{and} \quad%
V_u (x) := %
\nu_{x+1} - \nu^u_x +\la_x - \la^u_x.
\]
Let $\vphi\dvtx
\dR\to\dR$ be a smooth convex function such that for some constant
$c>0$, and
for all $r\in\dR$,
%
\begin{equation}\label{eqphi}
\vphi'(r)r \geq c\vphi(r).
\end{equation}
In particular, the behavior at infinity is at least polynomial of
degree $c$.

Let us state our first main result about intertwining and sub-commutation
relations between the original process $(X_t)_{t\geq0}$ and its
$u$-modification $(X_{u, t})_{t\geq0}$. To the knowledge of the
authors, this
result was not known. A connection to Chen's results on birth--death processes
\cite{chen2} is given in Section \ref{sectappli} in the sequel.

\begin{theo}[(Intertwining and sub-commutation)]\label{theocommutation}
Assume that the process is irredu\-cible, nonexplosive and that the potential
$V_u$ is lower bounded. Let $f \in\cF$ be such that $\sup_{y\in\dN
} \vert
\pd_u f (y) \vert<\infty$, and let $x\in\dN$ and $t\geq0$. Then the
following intertwining relation holds:
%
\begin{equation}\label{eqintert}
\pd_u P_t f (x) = P_{u,t}^{V_u} \,\pd_u f (x) = \dE
_x \biggl[ \pd_u f(X_{u,t}) \exp\biggl( - \int_0^t V_u (X_{u,s})\,
\mathrm{d}s \biggr) \biggr].
\end{equation}
Moreover, if $V_u\geq0$ then we have the sub-commutation relation
%
\begin{equation}\label{eqfeynmancommutation}
\vphi( \pd_u P_t f )(x) %
\leq\dE_x \biggl[ \vphi( \pd_u f) (X_{u,t}) %
\exp\biggl( - \int_0^t c V_u ( X_{u,s})\, \mathrm{d}s \biggr) \biggr] .
\end{equation}
\end{theo}

\begin{pf}
The key point is the following intertwining relation
%
\begin{equation}\label{eqcommutu}
\pd_u \cL= \cL_u^{V_u} \,\pd_u,
\end{equation}
where $\cL_u$ is the generator of the $u$-modification process
$(X_{u,t})_{t\geq0}$ and $\cL_u^{V_u} :=\cL_u-V_u$ is the discrete
Schr\"odinger operator with potential $V_u$. Note that the relation
\eqref{eqcommutu} is somewhat similar to \eqref{eqLegalite} and follows
by simple computations. To prove \eqref{eqintert} from \eqref{eqcommutu},
we proceed as we did to obtain \eqref{eqegalite} from \eqref{eqLegalite}.
If we define
\[
s\in[0,t]\quad\mapsto\quad J(s) := P_{u,s}^{V_u}\, \pd_u P_{t-s} f ,
\]
then \eqref{eqintert} rewrites as $J(0) = J(t)$. Hence, it suffices to show
that $J$ is constant. By \cite{chen1}, we know that if $\pd_u f$ is bounded
then $\pd_u P_{t-s} f$ is also bounded. Hence, using the Kolmogorov equations
and \eqref{eqcommutu}, we obtain
\[
J'(s) = P_{u,s} ^{V_u} ( \cL_u^{V_u}\, \pd_u P_{t-s} f - \pd_u \cL
P_{t-s}f ) \\
= 0,
\]
yielding to the intertwining relation \eqref{eqintert}.

Now let us
prove the sub-commutation relation \eqref{eqfeynmancommutation}
by adapting the previous interpolation method, under the additional
assumption $V_u \geq0$. Denoting
\[
s\in[0,t]\quad\mapsto\quad J_c(s) := P_{u,s}^{c V_u} \varphi(\pd_u P_{t-s} f ),
\]
then \eqref{eqfeynmancommutation} rewrites as $J_c(0) \leq J_c(t)$. Hence
let us show that $J_c$ is a nondecreasing function. Since $\varphi(\pd_u
P_{t-s} f )$ is bounded, we have by the Kolmogorov equations:
\[
J_c '(s) = P_{u,s} ^{cV_u} (T),
\qquad\mbox{where }
T = \cL_u^{cV_u} \vphi(\pd_u P_{t-s} f) - \vphi'(\pd_u P_{t-s}
f)\, \pd_u \cL P_{t-s}f .
\]
Letting $g_u = \pd_u P_{t-s} f$, we obtain, by using \eqref{eqcommutu},
\begin{eqnarray*}
T & =& \cL_u^{cV_u} \vphi(g_u) - \vphi'(g_u) \cL_u^{V_u} g_u \\
  & =& \la^u \bigl( \pd\vphi(g_u) %
- \vphi'(g_u)\,\pd g_u \bigr) %
+ \nu^u \bigl( \pd^* \vphi(g_u) %
- \vphi'(g_u)\,\pd^* g_u \bigr) %
+ V_u \bigl( \vphi'(g_u) g_u %
- c\vphi(g_u) \bigr) \\
&=&\la^u A^\vphi(g_u,\pd g_u)+\nu^u A^\vphi(g_u,\pd^* g_u)
+ V_u \bigl( \vphi'(g_u) g_u %
- c\vphi(g_u) \bigr),
\end{eqnarray*}
where $A^\vphi(r,s)=\vphi(r+s)-\vphi(r)-\vphi'(r)s$ is the so-called
$A$-transform of $\vphi$ studied in \cite{chafai2} also known in convex
analysis as the Bregman divergence associated to $\vphi$ \cite{bregman}.
Note that $g_u+\pd g_u=g_u(\cdot+1)$ and $g_u+\pd^* g_u=g_u(\cdot
-1)$. Now,
since $\vphi$ is convex, we have $A^\vphi\geq0$. Moreover, using
\eqref{eqphi} and $V_u\geq0$ we obtain that $T\geq0$. Finally, we
get the
desired result since the Feynman--Kac semigroup $(P_{u,t}^{cV_u})_{t\geq0}$
is positivity preserving.
\end{pf}

\begin{remark}[(Ergodic condition)]\label{remergodic}
The potential $V_u$ in Theorem \ref{theocommutation} is assumed to be lower
bounded. When it is positive, the so-called Chen exponent $\inf_{y\in
\dN}
V_u (y)$ is related to the exponential ergodicity of the original process
$(X_t)_{t\geq0}$, cf. \cite{chen1}. However, identity $(\ref{eqintert})$
does not require such an ergodic assumption. A nice study of the exponential
decay of birth--death processes was recently studied by Chen in \cite{chen3},
with special emphasis on nonergodic situations including transient
cases.
\end{remark}

\begin{remark}[(Case of equality)]\label{remergodic}
According to the proof of Theorem \ref{theocommutation}, the assumption
$V_u \geq0$ can be dropped if the convex function $\varphi$ realizes the
equality in \eqref{eqphi}. Such an observation was expected since in this
case the use of H\"older's inequality in \eqref{eqintert} entails the
desired result.
\end{remark}

\begin{remark}[(Propagation of monotonicity)]\label{remmonotone}
The identity \eqref{eqintert} provides a new proof of the propagation of
monotonicity \cite{stoyan}, Proposition 4.2.10: if $f \in\cF_{d}$ then
$P_tf \in\cF_{d}$ for all $t\geq0$. See Section \ref{sectconvdom}
for an
interpretation in terms of stochastic ordering.
\end{remark}

\begin{remark}[(Other gradients)]
Theorem \ref{theocommutation} possesses a natural analogue for the discrete
backward gradient $\pd^*$. We ignore if there exists a useful ``balanced''
intertwining relation involving a combination of both forward and backward
gradients.
\end{remark}

\begin{remark}[(Higher dimensional spaces)]
The extension of Theorem \ref{theocommutation} to higher dimensional
discrete processes such as queuing networks or interacting particles
systems arising in statistical mechanics is a very natural question, but
seems to be technically difficult. However, a first step has been emphasized
by Wu in his study of functional inequalities for Gibbs measures
through the
Dobrushin uniqueness condition: see step 1 in the proof of Proposition 2.5
in \cite{wu}.
\end{remark}

Our second new result below complements the previous one for the case $u=1$.
Let $\mathcal{I}$ be an open interval of $\dR$ and let
$\varphi\dvtx\mathcal{I}\to\dR$ be a smooth convex function such that
$\varphi''
>0$ and $-1/\varphi''$ is convex on $\mathcal{I}$. Following the
notations of
\cite{chafai2}, we define on the convex subset $\cA_\cI:= \{ (r,s)
\in\dR^2
\dvt (r,r+s) \in\cI\times\cI\}$ the nonnegative function $B^\vphi$ on
$\cA_\cI$ by
\[
B^\vphi(r,s) := \bigl( \vphi'(r+s)-\vphi'(r) \bigr) s,\qquad (r,s) \in
\cA_\cI.
\]
By Theorem 4.4 in \cite{chafai2}, $B^\vphi$ is convex on $\cA_\cI$. Some
interesting examples of such functionals will be given in
Section \ref{sectFI} below.
%
\begin{theo}[(Sub-commutation for $\boldsymbol{1}$-modification)]\label{theobicommutation}
Assume that the process is irreducible and nonexplosive.
If the transition rate $\la$ is
nonincreasing and $\nu$ is nondecreasing then for any function $f \in
\cF$
such that $\sup_{y\in\dN} \vert\pd f(y) \vert<\infty$ and for any
$t\geq0$,
%
\begin{equation}\label{eqbicommutation}
B^\vphi( P_t f , \pd P_t f ) \leq P_{1,t}^{V_1} B^\vphi(f,\pd f),
\end{equation}
where the nonnegative potential is $V_1 := \pd(\nu-\la)$.
\end{theo}

\begin{pf}
Under our assumption, the two processes $(X_t)_{t\geq0}$ and
$(X_{1,t})_{t\geq0}$ are nonexplosive. By using standard approximation
procedures, one may assume that $f$ has finite support. If we define
$s\in[0,t]\mapsto J (s) := P_{1,s}^{V_1} B^\vphi(P_{t-s} f,\pd
P_{t-s} f)$
we see that \eqref{eqbicommutation} rewrites as $J(0) \leq J(t)$.
Denote $F
= P_{t-s} f$ and $G = \pd P_{t-s} f = \pd F$. Since $ B^\vphi(F,G)$ is
bounded, the Kolmogorov equations are available and using
\eqref{eqcommutu} with the constant function $u=1$, we have $J'(s) =
P_{1,s}^{V_1} (T)$ with
\begin{eqnarray*}
T & = & \cL_1^{V_1} B^\vphi(F,G) %
- \frac{\pd}{\pd x} B^\vphi(F,G) \cL F %
- \frac{\pd}{\pd y} B^\vphi(F,G) \cL_1^{V_1} G \\
& = &\la^1 \,\pd B^\vphi(F,G) %
- \la\frac{\pd}{\pd x} B^\vphi(F,G)\, \pd F %
- \la^1 \frac{\pd}{\pd y} B^\vphi(F,G)\, \pd G \\
&&{} + \nu^1\, \pd^* B^\vphi(F,G) %
- \nu\frac{\pd}{\pd x} B^\vphi(F,G)\, \pd^* F %
- \nu^1 \frac{\pd}{\pd y} B^\vphi(F,G)\, \pd^* G \\
&&{} + \pd(\nu-\la) \biggl( \frac{\pd}{\pd y} B^\vphi(F,G) G %
- B^\vphi(F,G) \biggr) \\
& \geq& \pd\nu\biggl( \frac{\pd}{\pd y} B^\vphi(F,G) G %
- B^\vphi(F,G) \biggr) \\
&&{} - \pd\la\biggl( \frac{\pd}{\pd y} B^\vphi(F,G) G %
- \frac{\pd}{\pd x} B^\vphi(F,G) G - B^\vphi(F,G) \biggr) ,
\end{eqnarray*}
and where in the last line we used the convexity of the bivariate function
$B^\vphi$. Moreover, since the birth and death rates $\la$ and $\nu$ are
respectively, nonincreasing and nondecreasing on the one hand, and using
once again convexity on the other hand, we get
\[
\frac{\pd}{\pd y} B^\vphi(F,G) G \geq
\cases{\displaystyle
\frac{\pd}{\pd x} B^\vphi(F,G) G + B^\vphi(F,G), \vspace*{2pt}\cr
 B^\vphi(F,G)
}
\]
from which we deduce that $T$ is nonnegative and thus $J$ is
nondecreasing.
\end{pf}

\begin{remark}[(Diffusion case)]\label{remdiff}
Actually, the intertwining relations above have their counterpart in
continuous state space, as suggested by the so-called Witten Laplacian
method used for the analysis of Langevin-type diffusion processes, see for
instance Helffer's book \cite{helffer}. Let $\cA$ be the generator of a
one-dimensional real-valued diffusion $(X_{t})_{t\geq0}$ of the type
\[
\cA f = \si^2 f''+ bf',
\]
where $f$ and the two functions $\si,b$ are sufficiently smooth. Given a
smooth positive function $a$ on $\dR$, the gradient of interest is
$\nabla_a
f = a f'$. Denote $(P_t)_{t\geq0}$ the associated diffusion semigroup.
Then it is not hard to adapt to the continuous case the argument of
Theorem \ref{theocommutation} to show that the following intertwining
relation holds:
\[
\nabla_a P_tf (x) %
= \dE_x \biggl[ \nabla_a f(X_{a,t}) \exp\biggl( - \int_0^t V_a ( X_{a,s})
\, \mathrm{d}s \biggr) \biggr] .
\]
Here $(X_{a,t})_{t\geq0}$ is a new diffusion process with generator
\[
\cA_a f = \si^2 f'' + b_a f'
\]
and drift $b_a$ and potential $V_a$ given by
\[
b_a := 2\si\si' +b - 2\si^2 \frac{a'}{a}
\quad\mbox{and}\quad
V_a := \si^2 \frac{a''}{a} - b' + \frac{a'}{a} b_a.
\]
In particular, if the weight $a=\si$, where $\si$ is assumed to be
positive, then the two processes above have the same distribution and by
Jensen's inequality, we obtain
\[
\vert\nabla_\si P_tf (x) \vert%
\leq\dE_x \biggl[ \vert\nabla_\si f (X_{t}) \vert%
\exp\biggl( - \int_0^t %
\biggl( \si\si'' -b' + b \frac{\si'}{\si} \biggr)%
( X_{s})\, \mathrm{d}s \biggr) \biggr] .
\]
Hence under the assumption that there exists a constant $\rho$ such that
\[
\inf\si\si'' -b' + b \frac{\si'}{\si} \geq\rho,
\]
then we get $\vert\nabla_\si P_tf \vert\leq \mathrm{e}^{-\rho t} P_t \vert
\nabla_\si f \vert$. This type of sub-commutation relation is at the heart
of the Bakry--\'Emery calculus \cite{bakemery,bakry2,ledoux}. See also
\cite{malrieutalay} for a nice study of functional inequalities for the
invariant measure under the condition $\rho=0$. However, as we will
see in
Remark \ref{remdiff2} below, such a choice of the weight is not really
adapted when studying the optimal constant in the Poincar\'e inequality.
\end{remark}

\section{Applications}\label{sectappli}

This section is devoted to applications of Theorems \ref
{theocommutation} and
\ref{theobicommutation}.

\subsection{Lipschitz contraction and Wasserstein curvature}\label{sectwass}

Theorem \ref{theocommutation} allows to recover a result of Chen
\cite{chen1} on the contraction property of the semigroup on the space of
Lipschitz functions. Indeed, the intertwining (\ref{eqintert}) can be
used to
derive bounds on the Wasserstein curvature of the birth--death process, without
using the coupling technique emphasized by Chen. For a distance $d$ on
$\dN$,
we denote by $\cP_d (\dN)$ the set of probability measures $\xi$ on
$\dN$
such that $\sum_{x\in\dN} d(x,x_0)\xi(x) <\infty$ for some (or equivalently
for all) $x_0 \in\dN$. We recall that the Wasserstein distance
between two
probability measures $\mu_1, \mu_2 \in\cP_d (\dN)$ is defined by
%
\begin{equation}\label{eqdistwass}
\cW_d (\mu_1 , \mu_2) %
= \inf_{\gamma\in\Marg(\mu_1, \mu_2)}\int_\dN\int_\dN d(x,y)
\gamma(\mathrm{d}x,\mathrm{d}y) ,
\end{equation}
where $\Marg(\mu_1, \mu_2)$ is the set of probability measures on
$\dN^2$ such
that the marginal distributions are $\mu_1$ and $\mu_2$,
respectively. The
Kantorovich--Rubinstein duality \cite{villani}, Theorem 5.10, gives
%
\begin{equation}\label{eqKR}
\cW_d (\mu_1,\mu_2) =\sup_{g\in\Lip_1 (d)} \int_\dN g\, \mathrm{d}(\mu_1 -
\mu_2) ,
\end{equation}
where $\Lip(d)$ is the set of Lipschitz function $g$ with respect to the
distance $d$, that is,
\[
\Vert g \Vert_{\Lip(d)} %
:=\mathop{\sup_{x,y \in\dN}}_{x\neq y} \frac{\vert
g(x)-g(y)\vert}{d(x,y)}
<\infty,
\]
and $\Lip_1 (d)$ consists of 1-Lipschitz functions. We assume that the kernel
$P_t^x \in\cP_d (\dN)$ for every $x\in\dN$ and $t\geq0$ so that the
semigroup is well-defined on $\Lip(d)$. The Wasserstein curvature of
$(X_t)_{t\geq0}$ with respect to a given distance $d$ is the optimal
(largest) constant $\si$ in the following contraction inequality:
%
\begin{equation}\label{eqcontractlip}
\Vert P_t \Vert_{\Lip(d) \to\Lip(d)} \leq \mathrm{e}^{-\si t } ,\qquad t\geq0.
\end{equation}
Here $\Vert P_t \Vert_{\Lip(d) \to\Lip(d)}$ denotes the supremum of
$\Vert
P_t f \Vert_{\Lip(d)}$ when $f$ runs over $\Lip_1 (d)$. It is actually
equivalent to the property that
\[
\cW_d (P_t ^x,P_t ^y) \leq \mathrm{e}^{-\si t } d(x,y) ,\qquad %
x,y \in\dN, t\geq0.
\]
If the optimal constant is positive, then the process is positive recurrent
and the semigroup converges exponentially fast in Wasserstein distance
$\cW_{d}$ to the stationary distribution $\mu$ \cite{chen2},
Theorem 5.23.

Let $\rho\in\cF_{+}$ be an increasing function and define $u \in
\cF_{+}$ as $u_x := \rho(x+1) - \rho(x)$. The metric under consideration
in the forthcoming analysis is
\[
d_u (x,y) = \vert\rho(x) - \rho(y) \vert.
\]
Hence, $u$ remains for the distance between two consecutive points. In
particular, the space of functions $f$ for which the intertwining
relation of
Theorem \ref{theocommutation} is available is actually $\Lip(d_u)$.
Then it
is shown in \cite{chen1,joulin} by coupling arguments that the Wasserstein
curvature $\si_u$ with respect to the distance $d_u$ is given by the Chen
exponent, that is,
\[
\si_u %
= \inf_{x\in\dN} \nu_{x+1} %
- \nu_x \frac{u_{x-1}}{u_x} %
+ \la_x %
- \la_{x+1} \frac{u_{x+1}}{u_x} .
\]
The following corollary of Theorem \ref{theocommutation} allows to recover
this result via an intertwining relation.

\begin{corol}[(Contraction and curvature)]\label{corolwass}
Assume that the potential $V_u$ is lower bounded. Then
with the notations of Theorem \ref{theocommutation}, for any $t\geq0$,
%
\begin{equation}\label{eqcontraction}
\Vert P_t \Vert_{\Lip(d_u) \to\Lip(d_u)} = \Vert P_t \rho\Vert
_{\Lip(d_u)} = \sup_{x\in\dN} \dE_x \biggl[ \exp\biggl( - \int
_0^t V_u (X_{u ,s})\, \mathrm{d}s \biggr) \biggr] .
\end{equation}
In particular, the contraction inequality \eqref{eqcontractlip} is
satisfied with the optimal constant
%
\begin{equation}\label{eqcheninf}
\si_u=\inf_{y\in\dN} V_{u}(y).
\end{equation}
\end{corol}

\begin{pf}
Let $f \in\Lip_1 (d_u)$ be a 1-Lipschitz function with respect to the
distance $d_u$. For any $y,z \in\dN$ such that $y<z$ (without loss of
generality), we have by the intertwining identity \eqref{eqintert} of
Theorem \ref{theocommutation} and Jensen's inequality,
\begin{eqnarray*}
\vert P_t f(z)-P_t f(y) \vert
& \leq&\sum_{x=y}^{z-1} u_x \vert\pd_u P_t f (x) \vert\\
& \leq&\sum_{x=y}^{z-1} u_x \dE_x \biggl[ \vert\pd_u f(X_{u,t}) \vert
\exp\biggl( - \int_0^t V_{u} (X_{u,s})\, \mathrm{d}s \biggr) \biggr] \\
& \leq& d_u(z,y) \sup_{x\in\dN} \dE_x \biggl[ \exp\biggl( - \int_0^t V_{u}
(X_{u,s})\, \mathrm{d}s \biggr) \biggr] ,
\end{eqnarray*}
so that dividing by $d_u(z,y)$ and taking suprema entail the inequality:
\[
\Vert P_t \Vert_{\Lip(d_u) \to\Lip(d_u) } %
\leq\sup_{x\in\dN} %
\dE_x \biggl[ \exp\biggl( - \int_0^t V_{u} (X_{u,s})\,  \mathrm{d}s \biggr) \biggr].
\]
Finally, since by Remark \ref{remmonotone} the semigroup $(P_t)_{t\geq0}$
propagates monotonicity, the right-hand side of the latter inequality is
nothing but $\Vert P_t \rho\Vert_{\Lip(d_u)} $, showing that the supremum
over $\Lip_1 (d_u)$ is attained for the function $\rho$. The proof of
equation \eqref{eqcontraction} is achieved.

To establish \eqref{eqcheninf}, note that it suffices to get part
$\leq$
since the other inequality follows from \eqref{eqcontraction}. Applying
\eqref{eqintert} to the function $\rho$ which is trivially in $ \Lip_1
(d_u)$, we have for all $x\in\dN$,
\[
\si_u \leq- \frac{1}{t} \log\dE_x \biggl[ \exp\biggl( - \int_0^t V_u (X_{u,s})
\, \mathrm{d}s \biggr) \biggr] ,\qquad t\geq0,
\]
and taking the limit as $t\to0$ entails the inequality $\si_u \leq V_u
(x)$, available for all $x\in\dN$. The proof of \eqref{eqcheninf} is now
complete.
\end{pf}

\begin{remark}[(Pointwise gradient estimates for the Poisson equation)]
The argument used in the proof of Corollary \ref{corolwass} allows
also to
obtain pointwise gradient estimates for the solution of the Poisson equation
at the heart of Chen--Stein methods \cite
{barbour,brownxia,barbourxia,schuhmacher}. More precisely, let us
assume that $d_u$ is such that $\rho
\in L^1 (\mu)$. For any centered function $f\in\Lip_1 (d_u)$, let us
consider the Poisson equation $-\cL g = f$, where the unknown is $g$. Then
under the assumption $\si_u >0$, there exists a unique centered solution
$g_f \in\Lip(d_u)$ to this equation given by the formula $g_f =
\int_0^\infty P_t f\, \mathrm{d}t$. We have for any $x\in\dN$ the following
estimate (compare with \cite{liuma}, Theorem 2.1):
\begin{eqnarray*}
\sup_{f \in\Lip_1 (d_u) } \vert\pd g_f (x)\vert
& = &\sup_{f \in\Lip_1 (d_u) } u_x \int_0^\infty\vert\pd_{u} P_t
f(x) \vert\, \mathrm{d}t\\
& =& u_x \int_0^\infty\pd_{u} P_t \rho(x)\, \mathrm{d}t \\
& =& u_x \int_0^\infty\dE_x \biggl[ \exp\biggl(-\int_0^t V_{u} (X_{u,s})\, \mathrm{d}s \biggr)
\biggr]\, \mathrm{d}t \\
& \leq&\frac{u_x}{\si_u} .
\end{eqnarray*}
\end{remark}

\subsection{Functional inequalities}\label{sectFI}

Theorems \ref{theocommutation} and \ref{theobicommutation} allow to
establish a whole family of discrete functional inequalities. We define the
bilinear symmetric form $\Gamma$ on $\cF$ by
\[
\Ga(f,g) := \tfrac{1}{2}{{\bigl(\cL(fg)-f\cL g -g\cL f\bigr)}} %
= \tfrac{1}{2}{{(\la\,\pd f\,\pd g + \nu\,\pd^* f \,\pd^* g)}}.
\]
Under the positive recurrence assumption, the associated Dirichlet form acting
on its domain $\cD(\cE_\mu) \times\cD(\cE_\mu) $ is given by
\[
\cE_\mu(f,g) := \frac{1}{2}\int_\dN\Ga(f,g)\, \mathrm{d}\mu%
= \int_\dN\la\,\pd f\,\pd g\, \mathrm{d}\mu,
\]
where the second equality comes from the reversibility of the process. Here
the domain $\cD(\cE_\mu)$ corresponds to the subspace of functions
$f\in
L^2(\mu)$ such that $\cE_\mu(f,f)$ is finite. The stationary distribution
$\mu$ is said to satisfy the Poincar\'e inequality with constant $c$
if for
any function $f \in\cD(\cE_\mu) $,
%
\begin{equation}\label{eqpoincare}
c \Var_\mu(f) \leq\cE_\mu(f,f),
\end{equation}
where $\Var_\mu(f) := \mu(f^2)-\mu(f)^2$ and $\mu(f):= \int_\dN f\,
\mathrm{d}\mu$.
The optimal (largest) constant $c_\rP$ is the spectral gap of $\cL$,
that is, the
first nontrivial eigenvalue of the operator $-\cL$. The constant
$c_\rP$
governs the $L^2(\mu)$ exponential decay to the equilibrium of the semigroup:
for all $f\in L^2(\mu)$ and $t\geq0$,
\[
\Vert P_t f - \mu(f) \Vert_{L^2(\mu)} %
\leq \mathrm{e}^{-c_\rP t } \Vert f - \mu(f) \Vert_{L^2(\mu)}.
\]
Several years ago, Chen used a coupling method which provides the following
formula for the spectral gap:
\[
c_\rP= \sup_{u\in\cF_{+}} \si_u,
\]
where $\si_u$ is the Wasserstein curvature of Section \ref{sectwass}
or, in
other words, the Chen exponent. It corresponds to Theorem 1.1 in \cite{chen1},
equation (1.4). The following corollary of Theorem \ref{theocommutation}
allows to recover the $\geq$ part of Chen's formula.

\begin{corol}[(Spectral gap and Wasserstein curvatures)]\label{corolpoincare}
Assume that there exists some function $u\in\cF_{+}$ such that the
associated Wasserstein curvature $\si_u $ is positive. Then the
Poincar\'e inequality \eqref{eqpoincare} holds with constant $\sup
_{u\in
\cF_{+}}\si_u$, or in other words
\[
c_\rP\geq\sup_{u\in\cF_{+}} \si_u.
\]
\end{corol}

\begin{pf}
Since there exists some function $u\in\cF_{+}$ such that the Wasserstein
curvature $\si_u$ is positive, the process is positive recurrent. By
Proposition 6.59 in \cite{chen2}, the subspace of $ \cD(\cE_\mu)$ consisting
of functions with finite support is a core of the Dirichlet form and
thus we
can assume without loss of generality that $f$ has finite support. We have
\begin{eqnarray*}
\Var_\mu(f)
&=& - \int_\dN\int_0^\infty\frac{\mathrm{d}}{\mathrm{d}t} (P_tf)^2 \, \mathrm{d}t \, \mathrm{d}\mu\\
&=& -2 \int_\dN\int_0^\infty P_tf \cL P_tf \, \mathrm{d}t \, \mathrm{d}\mu\\
&=& 2 \int_0^\infty\int_\dN\la u^2 (\pd_u P_tf )^2 \, \mathrm{d}\mu\, \mathrm{d}t \\
& \leq& 2 \int_0^\infty \mathrm{e}^{-2\si_u t} \int_\dN\la u^2
P_{u,t} (\pd_u f )^2 \,\mathrm{d}\mu\, \mathrm{d}t ,
\end{eqnarray*}
where in the last line we used Theorem \ref{theocommutation} with the
convex function $\vphi(x) = x^2$. Now the measure $\la u^2 \mu$ is
invariant for the semigroup $(P_{u,t})_{t\geq0}$, so that we have
\begin{eqnarray*}
\Var_\mu(f)
& \leq&2 \int_0^\infty \mathrm{e}^{-2\si_u t} \int_\dN\la u^2
(\pd_u f )^2 \,\mathrm{d}\mu \,\mathrm{d}t \\
& =& \frac{1}{\si_u} \int_\dN\la(\pd f )^2 \, \mathrm{d}\mu\\
& = &\frac{1}{\si_u} \cE_\mu(f,f) ,
\end{eqnarray*}
where in the second line we used $\si_u >0$. The proof of the Poincar\'e
inequality is complete.
\end{pf}

\begin{remark}[($M/M/\infty$ and $M/M/1$)]
The spectral gap of the $M/M/\infty$ and $M/M/1$ processes is well-known
\cite{chen1}. Corollary \ref{corolpoincare} allows to recover it easily.
Indeed, in the $M/M/\infty$ case, the value $c_\rP=\nu$ can be
obtained as
follows: choose the constant weight $u=1$ to get $c_\rP\geq\nu$, and
notice that the equality holds for affine functions. For a positive
recurrent $M/M/1$ process, that is, $\la<\nu$, we obtain $c_\rP\geq
(\sqrt{\la}-\sqrt{\nu})^2$ by choosing the weight $u_x := (\nu/\la
)^{x/2}$,
whereas the equality asymptotically holds in \eqref{eqpoincare} as
$\kappa
\to\sqrt{\nu/\la}$ for the functions $\kappa^x$, $x\in\dN$. We conclude
that $c_\rP=(\sqrt{\la}-\sqrt{\nu})^2$.
\end{remark}
%
\begin{remark}[(Alternative method for $M/M/1$)]\label{remaltern}
In the $M/M/1$ case, let us recover the bound $c_\rP\geq
(\sqrt{\la}-\sqrt{\nu})^2$ by using a different method. Letting
$\rho(x) :=
x$ for $x\in\dN$ and $g=f-f(0)$ for a given function $f \in\cD(\cE
_\mu)$, we
have
\begin{eqnarray*}
\int_\dN g ^2\, \mathrm{d}\mu
& =& \frac{1}{\nu- \la} \int_\dN g^2 (-\cL\rho) \, \mathrm{d}\mu\\
& =& \frac{1}{\nu- \la} \cE_\mu(g^2,\rho) \\
& =& \frac{\la}{\nu- \la} \int_\dN\pd(g^2)\, \pd\rho\, \mathrm{d}\mu\\
& =& \frac{\la}{\nu- \la} \int_\dN\bigl( 2g\, \pd f + (\pd f)^2 \bigr) \, \mathrm{d}\mu\\
& \leq&\frac{\la}{\nu- \la} \Biggl( 2 \sqrt{\int_\dN g^2\, \mathrm{d}\mu} \sqrt
{\int_\dN(\pd f)^2\, \mathrm{d}\mu} + \int_\dN(\pd f)^2\, \mathrm{d}\mu\Biggr),
\end{eqnarray*}
where in the last inequality we used Cauchy--Schwarz' inequality.
Solving this
polynomial of degree 2 entails the inequality
\[
\int_\dN g ^2 \, \mathrm{d}\mu
\leq
\frac{\la}{(\sqrt{\la} - \sqrt{\nu} )^2 } \int_\dN(\pd f)^2
\, \mathrm{d}\mu.
\]
Finally using the inequality $ \Var_\mu(f) \leq\int_\dN g ^2\, \mathrm{d}\mu
$, we get
the result.
\end{remark}
%
\begin{remark}[(Diffusion case)]\label{remdiff2}
As mentioned in Remark \ref{remdiff}, the argument above leading to the
Poincar\'e inequality might be extended to the positive recurrent diffusion
case. In particular, under the same notation we obtain the following lower
bound on the Poincar\'e constant
\[
c_\rP\geq\sup_{a} \inf_{x\in\dR} V_a (x) ,
\]
where the supremum is taken over all positive $\mathscr{C}^2 $
function $a$ on
$\dR$. Note that up to the transformation $a\to1/a$, such a formula was
already obtained by Chen and Wang in \cite{chenwang} through their
Theorem 3.1, equation $(3.4)$, by using a coupling approach somewhat similar
to that emphasized by Chen in the discrete case.
\end{remark}

Theorem \ref{theobicommutation} allows to derive functional
inequalities more
general than the Poincar\'e inequality. Let $\cI$ be an open interval
of $\dR$
and for a smooth convex function $\vphi\dvtx\cI\to\dR$ such that $\vphi
'' >0$ and
$-1/\vphi''$ is convex on $\cI$, we define the $\varphi$-entropy of a
sufficiently integrable function $f\dvtx\dN\to\cI$ as
\[
\Ent_\mu^\vphi(f) = \mu( \vphi(f) ) - \vphi( \mu(f) ).
\]
Following \cite{chafai1}, we say that the stationary distribution $\mu$
satisfies a $\varphi$-entropy inequality with constant $c>0$ if for any
$\cI$-valued function $f\in\cD(\cE_\mu)$ such that $\vphi'(f)\in
\cD(\cE_\mu)$,
%
\begin{equation}\label{eqphientropy}
c \Ent_\mu^\vphi(f) \leq\cE_\mu( f, \vphi'(f) ).
\end{equation}
See, for instance, \cite{chafai2} for an investigation of the
properties of
$\varphi$-entropies. The $\varphi$-entropy inequality \eqref{eqphientropy}
is satisfied if and only if the following entropy dissipation of the semigroup
holds: for any sufficiently integrable $\cI$-valued function $f$ and every
$t\geq0$,
\[
\Ent_\mu^\vphi(P_t f) \leq \mathrm{e}^{-ct} \Ent_\mu^\vphi(f).
\]
We have the following corollary of Theorem \ref{theobicommutation}.

\begin{corol}[(Entropic inequalities and Wasserstein curvature)]\label{corolphientropy}
If the birth rate $\la$ is nonincreasing and the Wasserstein curvature
$\si_1$ (with the constant weight $u=1$) is positive, then the
$\varphi$-entropy inequality \eqref{eqphientropy} holds with constant
$\si_1$.
\end{corol}

\begin{pf}
As in the proof of Corollary \ref{corolpoincare} the assertion $\si_1>0$
entails the positive recurrence of the process. Moreover, we assume once
again that the $\cI$-valued function $f$ has finite support. By
reversibility, we have
\begin{eqnarray*}
\Ent_\mu^\vphi(f)
& =& \int_\dN\bigl( \vphi(P_0 f) - \vphi(\mu(f)) \bigr)\, \mathrm{d}\mu\\
& =& - \int_\dN\int_0^\infty\frac{\mathrm{d}}{\mathrm{d}t} \vphi(P_t f)\, \mathrm{d}t \,\mathrm{d}\mu\\
& =& - \int_0^\infty\int_\dN\vphi'(P_t f) \cL P_t f \,\mathrm{d}\mu\, \mathrm{d}t \\
& =& \int_0^\infty\int_\dN\la\,\pd P_t f\, \pd\vphi'(P_t f)\, \mathrm{d}\mu \,\mathrm{d}t
\\
& =& \int_0^\infty\int_\dN\la B^\vphi( P_t f , \pd P_t f
) \,\mathrm{d}\mu\, \mathrm{d}t ,
\end{eqnarray*}
where $B^\vphi$ is as in Theorem \ref{theobicommutation} (the
identity $\pd
g \,\pd\vphi'(g) = B^\vphi(g,\pd g)$ comes from $g+\pd g=g(\cdot+1)$).
Using now Theorem \ref{theobicommutation} together with the invariance of
the measure $\la\mu$ for the 1-modification semigroup
$(P_{1,t})_{t\geq
0}$, we obtain
\begin{eqnarray*}
\Ent_\mu^\vphi(f)
& \leq&\int_0^\infty\int_\dN \mathrm{e}^{-\si_1 t} \la P_{1,t} B^\vphi( f
, \pd f ) \,\mathrm{d}\mu\, \mathrm{d}t \\
& = &\int_0^\infty\int_\dN \mathrm{e}^{-\si_1 t} \la B^\vphi( f , \pd f )
\,\mathrm{d}\mu\, \mathrm{d}t \\
& =& \frac{1}{\si_1} \int_\dN\la B^\vphi( f , \pd f )\, \mathrm{d}\mu\\
& =& \frac{1}{\si_1} \cE_\mu( f, \vphi'(f) ) .
\end{eqnarray*}
\upqed\end{pf}

\begin{remark}[(Examples of entropic inequalities)]
The constant in the $\vphi$-entropy inequality provided by
Corollary \ref{corolphientropy} is not optimal in general (compare for
instance with the Poincar\'e inequality of Corollary \ref{corolpoincare}
when $\vphi(r)=r^2 $ with $\cI=\dR$). The choice $\vphi(r) = r\log
r$ with
$\cI=(0,\infty)$ allows us to recover the modified $\log$-Sobolev inequality
of \cite{caputo}, Theorem 3.1: for any positive function $f\in\cD
(\cE_\mu)$
such that $\log f \in\cD(\cE_\mu)$,
%
\begin{equation}\label{eqmlsi}
\si_1 \Ent_\mu^\vphi(f) \leq\cE_\mu(f,\log f).
\end{equation}
Note that beyond this entropic inequality, it is proved in \cite{caputo}
that the entropy is convex along the semigroup (a careful reading of the
proof in \cite{caputo} suggests that it simply boils down to
commutation and
convexity of $A$ transforms!). For the $M/M/\infty$ process, the
estimate of
Corollary \ref{corolphientropy} is sharp since $\si_1 = \nu$ and the
equality in \eqref{eqmlsi} holds as $\alpha\to\infty$ for the function
$x\in\dN\mapsto \mathrm{e}^{\alpha x}$. Note that the $M/M/1$ process and its
invariant distribution, which is geometric, do not satisfy a modified
$\log$-Sobolev inequality. Another $\vphi$-entropy inequality of
interest is
that obtained when considering the convex function $\phi(r) := r^p$,
$p\in
(1,2]$, with $\mathcal{I}=(0,\infty)$: for any positive function
$f\in
\cD(\cE_\mu)$ such that $f^{p-1} \in\cD(\cE_\mu)$,
%
\begin{equation}\label{eqbeckner}
\mu(f^p) - \mu(f) ^p \leq\frac{p}{\si_1} \mathcal{E}_\mu(f,f^{p-1}).
\end{equation}
Such an inequality has been studied in \cite{bobkovtetali} in the case of
Markov processes on a finite state space and also in \cite{chafai2}
for the
$M/M/\infty$ queuing process. In particular, it can be seen as an
interpolation between Poincar\'e and modified $\log$-Sobolev
inequalities.
\end{remark}

Under the positive recurrence assumption, Theorem \ref{theocommutation}
implies also other type of functional inequalities such as discrete
isoperimetry and transportation-information inequalities. Given a positive
function $u$, we focus on the distance $d_u$ constructed in
Section \ref{sectwass}, where we assume moreover that $\rho\in
\cD(\cE_\mu)$, that is, $\la u^2$ is $\mu$-integrable or, in other
words, the
$u$-modification process $(X_{u,t})_{t\geq0}$ is positive recurrent. The
invariant measure $\mu$ is said to satisfy a weighted isoperimetric inequality
with weight $u$ and constant $h_u >0$ if for any absolutely continuous
probability measure $\pi$ with density $f \in\cD(\cE_\mu)$ with
respect to
$\mu$,
%
\begin{equation}\label{eqwassisop}
h_u \cW_{d_u}(\pi, \mu) %
\leq\int_\dN\la u \vert\pd f \vert \, \mathrm{d}\mu,
\end{equation}
where the Wasserstein distance $\cW_{d_u}$ is defined in \eqref{eqdistwass}
with respect to the distance $d_u$. The terminology of isoperimetry is
employed here because it is a generalization of the classical isoperimetry,
which states that the centered $L^1$-norm is dominated by an energy of
$L^1$-type. Indeed, if the weight $u$ is identically 1, then the
distance $d_1$
between two different points is at least 1, so that \eqref{eqwassisop}
entails
\[
2h_1 \int_\dN\vert f-1\vert \,\mathrm{d}\mu%
= h_1 \cW_{d}(\pi, \mu) %
\leq h_1 \cW_{d_1}(\pi, \mu) %
\leq\int_\dN\la\vert\pd f \vert \,\mathrm{d}\mu,
\]
where $d$ is the trivial distance 0 or 1. Note that the $L^1$-energy
emphasized above differs from the discrete version of the diffusion case,
since our discrete gradient does not derive from $\Ga$.

On the other hand, let us introduce the transportation-information
inequalities emphasized in \cite{glwy}. Let $\alpha$ be a continuous positive
and increasing function on $[0,\infty)$ vanishing at 0. The invariant measure
$\mu$ satisfies a transportation-information inequality with deviation
function $\alpha$ if for any absolutely continuous probability measure
$\pi$
with density $f$ with respect to $\mu$, we have
%
\begin{equation}\label{eqtransp}
\alpha( \cW_{d_u}(\pi,\mu) ) \leq\cI(\pi, \mu) ,
\end{equation}
where the so-called Fisher--Donsker--Varadhan information of $\pi$ with respect
to $\mu$ is defined as
\[
\cI(\pi, \mu) := \cases{
\mathcal{E} _\mu\bigl(\sqrt{f},\sqrt{f}\bigr) & \quad if  $\sqrt{f} \in\cD
(\cE_\mu)$; \cr
\infty & \quad otherwise.}
\]
Note that $\cI(\cdot,\mu)$ is nothing but the rate function
governing the
large deviation principle in large time of the empirical measure $L_t :=
t^{-1} \int_0 ^t \delta_{X_s}\, \mathrm{d}s$, where $\delta_x$ is the Dirac
mass at
point $x$. In other words, the Fisher--Donsker--Varadhan information
rewrites as
the variational identity \cite{chen2}, Theorem 8.8:
\[
\cI(\pi,\mu) = \sup_{V\in\cF_+} \int_{\dN} - \frac{\cL V}{V}
\mathrm{d}\pi.
\]
The interest of the transportation-information inequality resides in the
equivalence with the following tail estimate of the empirical measure
\cite{glwy}, Theorem 2.4: for any absolutely continuous probability measure
$\pi$ with density $f \in L^2(\mu)$ with respect to $\mu$, and any
$g \in
\Lip_1 (d_u)$,
\[
\dP_\pi\bigl( L_t (g) - \mu(g) >r \bigr) %
\leq\Vert f\Vert_{L^2(\mu)} \mathrm{e}^{-\alpha(r)} ,\qquad r>0, t>0.
\]
We have the following corollary of Theorem \ref{theocommutation}.
%
\begin{corol}[(Weighted isoperimetry and transportation-information
inequality)]\label{corolcheeger}
With the notations of Theorem \ref{theocommutation}, assume that the
process is positive recurrent and that the following quantity is
well defined:
\[
\kappa_u := \int_0^\infty%
\sup_{x\in\dN} %
\dE_x \biggl[ \exp\biggl( - \int_0^t V_{u} (X_{u,s})\, \mathrm{d}s \biggr) \biggr] %
\, \mathrm{d}t <\infty.
\]
Then the weighted isoperimetric inequality $(\ref{eqwassisop})$ is satisfied
with constant $h_u = 1/\kappa_u$.
If moreover there exists two constants $\varepsilon>0$ and $\theta>1$ such
that
%
\begin{equation}\label{eqlyap}
(1+\varepsilon) \la_x u_x^2 + (1+1/\varepsilon)\nu_x u_{x-1} ^2 %
\leq-a \bigl( \la_x (\theta-1) + \nu_x (1/\theta-1) \bigr)+ b,\qquad x\in\dN,
\end{equation}
where $a:=a_{\varepsilon,\theta} \geq0$ and $b:=b_{\varepsilon
,\theta} >0$
are two other constants depending on both $\varepsilon$ and $\theta$, then
the transportation-information inequality $(\ref{eqtransp})$ is satisfied
with deviation function
\[
\alpha(r) := \sup_{\varepsilon>0 ,\theta>1} \frac{\sqrt{b^2+2a
(r/\kappa_u)^2}-b }{2a} .
\]
\end{corol}
%
\begin{remark}[(The case of positive Wasserstein curvature)]
In particular, if the Wasserstein curvature $\si_u$ with respect to the
distance $d_u$ is positive, then the process is positive recurrent and
we have
\[
\si_u \cW_{d_u}(\pi, \mu) %
\leq\int_\dN\la u \vert\pd f \vert\, \mathrm{d}\mu%
\quad\mbox{and}\quad %
\alpha( \cW_{d_u}(\pi,\mu) ) \leq\cI(\pi, \mu) ,
\]
with the deviation function
\[
\alpha(r) := \sup_{\varepsilon>0,\theta>1} \frac{\sqrt{b^2+2a
(r\si_u)^2}-b}{2a} .
\]
\end{remark}

\begin{pf*}{Proof of Corollary \protect\ref{corolcheeger}}
For every $f,g\in\cD(\cE_\mu)$ we have, by reversibility,
%
\begin{eqnarray}\label{eqvariational}
\Cov_\mu(f,g)
& := &\int_\dN\biggl( g- \int_\dN g\, \mathrm{d}\mu\biggr) f \, \mathrm{d}\mu\nonumber\\
& \hspace*{2.8pt}= &\int_\dN\biggl( - \int_0^\infty\cL P_t g \, \mathrm{d}t \biggr) f \, \mathrm{d}\mu\nonumber
\\[-8pt]
\\[-8pt]
& \hspace*{2.8pt}=& \int_0^\infty\biggl( - \int_\dN P_t g \cL f \, \mathrm{d}\mu\biggr) \, \mathrm{d}t \nonumber\\
& \hspace*{2.8pt}= &\int_0^\infty\cE_\mu(P_t g , f) \, \mathrm{d}t .\nonumber
\end{eqnarray}
Now, for every probability measure $\pi\ll\mu$ with $\mathrm{d}\pi=f\, \mathrm{d}\mu$,
$f \in\cD(\cE_\mu)$, we get,
using (\ref{eqvariational}),
\begin{eqnarray*}
\cW_{d_u}(\pi, \mu)
& =& \sup_{g\in\Lip_1 (d_u)} \Cov_\mu(f,g) \\
& =& \sup_{g\in\Lip_1 (d_u)} \int_0^\infty\cE_\mu(P_t g , f) \, \mathrm{d}t \\
& =& \sup_{g\in\Lip_1 (d_u)} \int_0^\infty\int_\dN\la u \, \pd f
\,\pd_{u} P_t g \, \mathrm{d}\mu \, \mathrm{d}t \\
& =& \int_0^\infty\int_\dN\la u \vert\,\pd f \vert\,\pd_{u} P_t
\rho \, \mathrm{d}\mu \, \mathrm{d}t \\
& \leq&\int_0^\infty\sup_{x\in\dN} \dE_x \biggl[ \exp\biggl( -
\int_0^t V_{u} (X_{u,s}) \, \mathrm{d}s \biggr) \biggr] \, \mathrm{d}t \int_\dN\la
u \vert\pd f \vert \, \mathrm{d}\mu,
\end{eqnarray*}
where in the last inequality we used Theorem \ref{theocommutation}. This
concludes the proof of the weighted isoperimetric inequality.

Using now Cauchy--Schwarz inequality, reversibility and then (\ref{eqlyap})
with $V_\theta(x):= \theta^x$, $x\in\dN$,
\begin{eqnarray*}
\cW_{d_u}(\pi, \mu) & \leq&\kappa_u \sqrt{\cI(\pi,\mu)} \sqrt{
\int_\dN\la
u^2 \bigl( \sqrt{f(\cdot+1)}+\sqrt{f} \bigr) ^2 \, \mathrm{d}\mu} \\
& \leq&\kappa_u \sqrt{\cI(\pi,\mu)} \sqrt{ \int_\dN\bigl(
(1+\varepsilon)\la
u^2 + (1+1/\varepsilon)\nu u_{\cdot-1} ^2 \bigr) f \, \mathrm{d}\mu} \\
& \leq&\kappa_u \sqrt{\cI(\pi,\mu)} %
\sqrt{ \int_\dN\biggl( - a \frac{\cL V_\theta}{V_\theta} +b \biggr) f \, \mathrm{d}\mu}
\\
& \leq&\kappa_u \sqrt{\cI(\pi,\mu)} \sqrt{a \cI(\pi,\mu)+b},
\end{eqnarray*}
from which the desired transportation-information inequality holds.
\end{pf*}

\begin{remark}[($M/M/\infty$ and $M/M/1$ revisited)]
Corollary \ref{corolcheeger} exhibits optimal functional inequalities, at
least in the $M/M/\infty$ case and its stationary distribution, the Poisson
measure of mean $\la/\nu$. Choosing the weight $u=1$, we obtain the optimal
constant $\hbar_1=\nu$ in the isoperimetric inequality. Indeed,
Corollary \ref{corolcheeger} entails $\hbar_1 \geq\nu$, whereas the other
inequality is obtained by choosing $\pi$ a Poisson measure of different
parameter. For the transportation-information inequality, we recover
Theorem 2.1 in \cite{mawu} since the choice of $a:=
\theta(1+1/\varepsilon)/(\theta-1)$ and $b:=\la(1+\varepsilon+
(1+1/\varepsilon)\theta)$ allows us to obtain the deviation function
$\alpha
(r) := \la(\sqrt{1+\nu r/\la}-1)^2$, $r>0$. Note that it is optimal
in view of Example 4.5 in \cite{guilwu}: for any absolutely continuous
probability measure $\pi$ with square-integrable density with respect to
$\mu$,
\[
\lim_{t\to\infty} \frac{1}{t} \log\dP_\pi\biggl( \frac{1}{t} \int_0
^t X_s \, \mathrm{d}s - \frac{\la}{\nu} >r \biggr) = - \la
\biggl(\sqrt{1+\frac{\nu r}{\la}}-1 \biggr)^2,\qquad r>0.
\]
For the $M/M/1$ process, we have the following inequalities for the optimal
isoperimetric constant $\hbar_u$, with $u_x = (\nu/\la)^{x/2}$ (a quantity
that will appear again in Section \ref{secthit}):
\[
\bigl(\sqrt{\la}-\sqrt{\nu}\bigr)^2 \leq\hbar_u \leq
\bigl(\sqrt{\nu}-\sqrt{\la}\bigr)\sqrt{\nu} .
\]
To get the second inequality, we choose the density $f=(\nu/\la)(1-1
_{\{
0\}})$ and the $1$-Lipschitz test function $g= \rho$. In particular as the
ratio $\la/\nu$ is small, we obtain $\hbar_u \approx\nu$. However, we
ignore if such a process satisfies a transportation-information inequality.
\end{remark}

\subsection{Hitting time of the origin by the $M/M/1$ process}\label{secthit}

Recall that we consider the ergodic $M/M/1$ process ($\la<\nu$) for
which the
stationary distribution is geometric of parameter $\la/\nu$. Since the
process behaves as a random walk outside 0, the ergodic property relies
essentially on its behavior at point $0$. Using the notation of
Theorem \ref{theocommutation}, the intertwining relation \eqref{eqintert}
applied with a positive function $u$ entails the identity
\[
\pd_{u} P_t f (x) %
= \dE_x \biggl[ \pd_{u} f(X_t) \exp\biggl( - \int_0^t V_{u}
(X_{u,s}) \, \mathrm{d}s \biggr) \biggr],
\]
where the potential is given for every $x\in\dN$ by
\[
V_{u} (x) %
:= \nu- \frac{u_{x-1}}{u_x} \nu\mathbf{1}_{\{ x\neq0 \}} +\la-
\frac{u_{x+1}}
{u_x}\la.
\]
Following Robert \cite{robert}, the process $(X_t^y)_{t\geq0}$ is the
solution of the stochastic differential equation
%
\begin{equation}\label{eqsde}
X_0^y=y \quad\mbox{and}\quad
\mathrm{d}X_t^y =  \mathrm{d}N_t^{(\la)} - \mathbf{1}_{\{X_{t-}^y >0\}} \, \mathrm{d}N_t^{(\nu)} ,\qquad t>0,
\end{equation}
where $(N_t^{(\la)})_{t\geq0}$ and $(N_t^{(\nu)})_{t\geq0}$ are two
independent Poisson processes with parameter $\la$ and $\nu$, respectively.
Since the process is assumed to be positive recurrent, the hitting time
of 0,
\[
T_0^y := \inf\{ t>0 \dvt X_t^y = 0\}
\]
is finite almost surely. We have the following corollary of Theorem
\ref{theocommutation}.

\begin{corol}[(Hitting time of the origin for the ergodic $M/M/1$
process)]\label{corolhitting}
Given $x\in\dN$, consider a positive recurrent $M/M/1$ process $(X_t
^{x+1})_{t\geq0}$ starting at point $x+1$, and denote
$(X_{u,t}^x)_{t\geq
0}$ its $u$-modification process starting at point $x$, where
\[
u_x := \biggl(\frac{\nu}{\la} \biggr)^{{x}/{2}} \geq1.
\]
Then we have the following tail estimate: for any $t\geq0$,
\begin{eqnarray*}
\dP(T_0^{x+1} >t )
&=& u_x  \mathrm{e}^{-t ( \sqrt{\la} - \sqrt{\nu} )^2} %
\dE\biggl[\frac{1}{u(X_{u,t}^x)}%
\exp{{\biggl(-\sqrt{\la\nu}\int_0^t\mathbf{1}_{\{0\}}(X_{u,s}^x)\, \mathrm{d}s\biggr)}}
\biggr] \\
&\leq& u_x  \mathrm{e}^{-t ( \sqrt{\la} - \sqrt{\nu} )^2}.
\end{eqnarray*}
\end{corol}

\begin{pf}
Let us use a coupling argument. Let $(X_t ^{x})_{t\geq0}$ be a copy of
$(X_t^{x+1})_{t\geq0}$, starting at point $x$. We assume that it
constructed with respect to the same driving Poisson processes
$(N_t^{(\la
)})_{t\geq0}$ and $(N_t^{(\nu)})_{t\geq0}$ as the process
$(X_t^{x+1})_{t\geq0}$. Hence, the stochastic differential equation
\eqref{eqsde} satisfied by the two coupling processes entails that the
difference between $(X_t^{x+1})_{t\geq0}$ and $(X_t^x)_{t\geq0}$ remains
constant, equal to 1, until time $T_0^{x+1}$, the first hitting time of the
origin by $(X_t^{x+1})_{t\geq0}$. After time $T_0^{x+1}$, the
processes are
identically the same, so that the following identity holds:
\[
X_t^{x+1} = X_t^x + \mathbf{1}_{\{ T_0^{x+1} >t \}} ,\qquad t\geq0.
\]
Since the original process is assumed to be positive recurrent, the coupling
is successful, that is, the coupling time is finite almost surely.
Therefore, we
have for any function $f\in\Lip(d_1)$, where $d_1$ is the distance
$d_1(x,y) = \vert x-y\vert$,
\[
\pd P_t f (x)
= P_t f(x+1) - P_t f (x)
= \dE[ f(X_t^{x+1})-f(X_t^x) ]
= \dE\bigl[ \pd f(X_t^x) \mathbf{1}_{\{ T_0^{x+1} >t \}} \bigr]
\]
so that if we denote the function $\rho(x)=x$, we obtain
\[
\dP(T_0^{x+1} >t )
= \pd P_t \rho(x)
= u_ x \,\pd_{u} P_t \rho(x).
\]
Using now \eqref{eqintert} with the function $u$, we get
\[
\dP(T_0^{x+1} >t ) %
= u_x %
\dE{{\biggl[\frac{1}{u(X_{u,t}^x)}\exp{{\biggl(-\int_0 ^t V_{u}(X_{u,s}^x)\, \mathrm{d}s\biggr)}}\biggr]}},
\]
where $V_u := (\sqrt{\la}-\sqrt{\nu})^2 + \sqrt{\la\nu}\mathbf
{1}_{\{0\}}$.
\end{pf}

\begin{remark}[(Sharpness)] Using a completely different approach, Van Doorn
established in \cite{vandoorn}, through his Theorem 4.2 together with his
Example 5, the following asymptotics
\[
\lim_{t\to\infty} \frac{1}{t} \log\dP(T_0^{x+1} >t) %
= - \bigl(\sqrt{\la}-\sqrt{\nu}\bigr)^2 ,\qquad x\in\dN.
\]
Hence, one deduces that the exponential decay in the result of Corollary
\ref{corolhitting} is sharp. On the other hand, Proposition 5.4 in
\cite{robert} states that $T_0^{x+1}$ has exponential moment bounded as
follows:
\[
\dE\bigl[  \mathrm{e}^{(\sqrt{\la}-\sqrt{\nu})^2 T_0^{x+1}} \bigr] \leq\biggl(
\frac{\nu}{\la} \biggr)^{(x+1)/2} ,
\]
so that Chebyshev's inequality yields a tail estimate somewhat similar to
ours -- although with a worst constant depending on the initial point $x+1$.
\end{remark}

\begin{remark}[(Other approach)]
The proof of Corollary \ref{corolhitting} suggests also a martingale
approach. First, note that we have the identity
\[
-\nu\mathbf{1}_{\{ 0\}} = -\frac{\cL u}{u} - V_{u}
\]
which entails as in the previous proof and since $u \geq1$, the following
computations:
\begin{eqnarray*}
\dP(T_0^{x+1} >t )
& =& \pd P_t \rho(x) \\
& =& \dE\biggl[
\exp\biggl(
- \int_0^t \nu\mathbf{1}_{\{ 0\}} (X_{s}^x) \, \mathrm{d}s \biggr) \biggr] \\
& \leq&\dE\biggl[ u(X_t^x) \exp\biggl( - \int_0^t \biggl( \frac{\cL u}{u} + V_{u}
\biggr)(X_{s}^x) \, \mathrm{d}s \biggr) \biggr] \\
& \leq& u_x  \mathrm{e}^{-t (\sqrt{\la}-\sqrt{\nu})^2} ,
\end{eqnarray*}
since the process $(M_t^u)_{t\geq0}$ given by
\[
M_t^u := u(X_t^x) \exp\biggl(-\int_0^t \frac{\cL u}{u} (X_s^x) \, \mathrm{d}s \biggr) ,\qquad
t\geq0 ,
\]
is a supermartingale. Indeed, denoting
\[
Z_t^u := \exp\biggl(-\int_0^t \frac{\cL u}{u} (X_s^x) \, \mathrm{d}s \biggr),
\]
we have by Ito's formula:
\begin{eqnarray*}
d M_t^u
& =& Z_t^u \, \mathrm{d}u(X_t^x) + u(X_t^x) \, \mathrm{d}Z_t^u \\
& =& Z_t^u \bigl( \mathrm{d} M_t + \cL u (X_t^x)\, \mathrm{d}t\bigr) - u (X_t^x) \frac{\cL u }{u}
(X_t^x) Z_t^u \, \mathrm{d}t \\
& =& Z_t^u \, \mathrm{d} M_t,
\end{eqnarray*}
where $(M_t)_{t\geq0}$ is a local martingale. Therefore, the process $(M_t
^u)_{t\geq0}$ is a positive local martingale and thus a supermartingale.
\end{remark}

\subsection{Convex domination of birth--death processes}\label{sectconvdom}

Let $(X_t^x)_{t\geq0}$ be the $M/M/\infty$ process starting from
$x\in\dN$.
The Mehler-type formula \eqref{eqmehler} states that the random variable
$X_t^x$ has the same distribution as the independent sum of the variable
$X_t^0$, which follows the Poisson distribution of parameter $(\la/\nu)
(1-\mathrm{e}^{-\nu t})$, and a binomial random variable $B_t^{(x)}$ of
parameters $(x,\mathrm{e}^{-\nu t})$. By convention, $B_t^{(0)}$ is assumed to be 0. Hence, we
have for
any nonnegative function $f$ and any $x\in\dN$,
%
\begin{equation}\label{eqmehlerfunct}
\dE[ f(X_t^x) ] = \dE\bigl[ f\bigl(X_t^0 + B_t^{(x)}\bigr)\bigr],\qquad t\geq0.
\end{equation}
Such an identity can be provided by using the commutation relation
\eqref{eqegalite}. Indeed we have
\[
\dE[ f(X_t^{x+1}) ] = (1- \mathrm{e}^{-\nu t}) \dE[ f(X_t^{x})
] +  \mathrm{e}^{-\nu t } \dE[ f(X_t^x +1)] ,
\]
so that a recursive argument on the initial state provides the required
result. An interesting consequence of \eqref{eqmehlerfunct} appears in terms
of concentration properties. For instance, a straightforward computation
entails that for any $\theta\geq0$, we get the following inequality
on the
Laplace transforms
\[
\dE[ \mathrm{e}^{\theta X_t ^x } ] \leq\dE[ \mathrm{e}^{\theta N_t ^{x}} ] ,
\]
where $N_t ^{x}$ is a Poisson random variable with the same mean as
$X_t ^x$.
Therefore, using the exponential Chebyshev inequality entails an upper bound
on the tail of the centered random variable $X_t ^x - \dE[X_t ^x]$,
which is
sharp as $t\to\infty$ (recall that the stationary distribution is Poisson
with parameter $\la/\nu$).

Actually, one may ask if for a more general
birth--death process, the intertwining relation of type \eqref
{eqintert} may
imply a relation similar to \eqref{eqmehlerfunct}. This leads to the notion
of stochastic ordering.

Following the presentation enlighten by Stoyan in \cite{stoyan}, let
us start
with the classical notion of stochastic ordering for integer-valued random
variables. We say that $X$ is stochastically smaller than $Y$, and we
note $X
\leq_d Y$, if for any function $f\in\cF_d$,
\[
\dE[f(X)] \leq\dE[f(Y)] .
\]
Such a relation, as the convex domination introduced below, is a partial
ordering on the set of distribution functions. The interesting feature
of this
stochastic ordering resides in its characterization in terms of
coupling: we
have $X \leq_d Y$ if and only if there exist random variables $X_1$ and
$Y_1$, both defined on the same probability space and with the same
distribution as $X$ and $Y$, respectively, such that $\dP(X_1 \leq X_2
) =1$.
Moreover, it is equivalent to the following comparison between tails:
we have
$X \leq_d Y$ if and only $\dP(X \geq x ) \leq\dP(Y \geq x )$ for
any $x\in
\dR$. In other words, the random variable $X$ takes small values with
a higher
probability than $Y$ does.

Another stochastic ordering of interest is the convex ordering, or convex
domination. Denote $\cF_c$ the subset of $\cF_d$ consisting of nonnegative
nondecreasing convex functions, where in our discrete setting the convexity
of a function $f\dvtx\dN\to\dR$ is understood as $\pd^2 f\geq0$. We
say that
$X$ is convex dominated by $Y$, and we note $X \cleq Y$, if for any function
$f\in\cF_c$,
\[
\dE[f(X)] \leq\dE[f(Y)] .
\]
It is known to be equivalent to the inequality
\[
\dE[ (X-x) ^+ ] \leq\dE[ (Y-x) ^+ ] ,\qquad x\in\dR,
\]
where $a^+ := \max\{ a,0 \}$. Typically, one may deduce from the convex
domination concentration properties like a comparison of moments or Laplace
transforms as in the $M/M/\infty$ case above. Moreover, this refined ordering
might appear for instance when using de-la-Vall\'ee-Poussin's lemma about
uniform integrability of a family of random variables. However, in
contrast to
the $\leq_d $ ordering, the authors ignore if there exists a genuine
interpretation of the convex domination in terms of coupling.

Coming back to our birth--death framework, we observe that if we want to use
the intertwining relation \eqref{eqintert} of Theorem \ref{theocommutation}
in order to obtain stochastic domination, then a first difficulty arises.
Indeed, another birth--death process appears in the right-hand-side of
\eqref{eqintert}, namely the $u$-modification of the original process.
Therefore, let us provide first a lemma which allows us to compare two
birth--death processes with respect to the $\leq_d$ ordering. Although the
result below is somewhat obvious from the point of view of coupling, we give
an alternative proof based on the interpolation method emphasized in
the proof
of Theorem \ref{theocommutation}. See also \cite{stoyan},
Proposition 4.2.10.

\begin{lemme}[(Stochastic comparison of birth--death processes)]\label{lemmedomin}
Let $(X_t^x)_{t\geq0}$ and $(\tilde{X_t^x})_{t\geq0}$ be two birth--death
processes both starting from $x\in\dN$. Denoting respectively $\la,
\nu$
and $\tilde{\la}, \tilde{\nu}$ the transition rates of the associated
generators $\cL$ and $\tilde{\cL}$, we assume that they satisfy the
following assumption:
\[
\tilde{\la} \leq\la\quad\mbox{and} \quad\tilde{\nu} \geq\nu.
\]
Then for every $t\geq0$, the random variable $\tilde{X_t^x}$ is
stochastically smaller than $X_t^x$. In other words, we have $\tilde{X_t^x}
\leq_d X_t^x$.
\end{lemme}

\begin{pf}
Let $g\in\cF_d$ and define the function
$s\in[0,t]\mapsto J(s) := \tilde{P_s} P_{t-s}g$ where $(P_t)_{t\geq
0}$ and
$(\tilde{P_t})_{t\geq0}$ are the semigroups of $(X_t^x)_{t\geq0}$ and
$(\tilde{X_t^x})_{t\geq0}$, respectively. By differentiation, we have
\[
J'(s)
= \tilde{P_s} ( \tilde{\cL} P_{t-s} g- \cL P_{t-s} g )
= \tilde{P_s} \bigl( (\tilde{\la} - \la)\, \pd P_{t-s} g %
+ (\tilde{\nu}-\nu)\, \pd^* P_{t-s} g \bigr),
\]
which is nonpositive since the semigroup $(P_t)_{t\geq0}$ satisfies the
propagation of monotonicity, cf. Remark \ref{remmonotone}. Hence, the
function $J$ is nonincreasing and the desired result holds.
\end{pf}

Now we are able to state the following corollary of Theorem
\ref{theocommutation}, which states a new convex domination involving
decoupled random variables in the right-hand side. However, despite some
particular cases like the $M/M/1$ case for which the convenient coupling
appearing in the proof of Corollary \ref{corolhitting} allows us to extend
the next result to the $\leq_d$ ordering, we ignore if it can be done
in full
generality.

\begin{corol}[(Convex domination)]
Denote $(X_t^y)_{t\geq0}$ a birth--death process starting at some point $y
\in\dN$. We assume that the birth rate $\la$ is nonincreasing and that
there exists $\kappa\geq0$ such that
\[
\pd(\nu-\la) \geq\kappa.
\]
Then for any $t\geq0$ and any $x\in\dN$, the random variable $X_t^{x+1}$
is convex dominated by the independent sum of $X_t ^x$ and a Bernoulli
random variable $Y_t$ of parameter $\mathrm{e}^{-\kappa t} \in(0,1]$. In other
words, we have
\[
X_t^{x+1} \cleq X_t^x + Y_t .
\]
\end{corol}

\begin{pf}
We have to show that for any function $f\in\cF_c$,
%
\begin{equation}\label{eqrecurrence}
\dE{{[f(X_t^{x+1})]}} \leq\dE{{[f(X_t^x + Y_t)]}}.
\end{equation}
Using the intertwining relation \eqref{eqintert} of
Theorem \ref{theocommutation}, we have since $f$ is nondecreasing:
\begin{eqnarray*}
\dE{{[f(X_t^{x+1})]}}
& \leq&\dE{{[f(X_t^{x})]}}+ \mathrm{e}^{-\kappa t } \dE{{[\pd f(X_{1,t}^x)]}}
\\
& \leq&\dE{{[f(X_t^{x})]}} + \mathrm{e}^{-\kappa t } \dE{{[\pd f(X_t^x)]}}
\\
& = &(1-\mathrm{e}^{-\kappa t})\dE{{[f(X_t^{x})]}} %
+\mathrm{e}^{-\kappa t }\dE{{[f(X_t^x+1)]}} \\
& = &\dE{{[f(X_t^x + Y_t)]}} ,
\end{eqnarray*}
where to obtain the second inequality we used Lemma \ref{lemmedomin} with
the 1-modification process $(X^x _{1,t})_{t\geq0}$ playing the role of
$(\tilde{X_t^x})_{t\geq0}$ since $\pd f$ is nondecreasing (recall
that \mbox{$f
\in\cF_c$}).
\end{pf}

\begin{remark}[(More on convex domination)]\label{remconvex0}
By an easy recursive argument one obtains from the latter result the
following convex domination:
\[
X_t^x \cleq X_t^0 + B_t^{(x)} ,
\]
where $B_t^{(x)}$ is a binomial random variable of parameters
$(x,\mathrm{e}^{-\kappa
t})$, independent from $X_t^0$, as in the case of the $M/M/\infty$
queuing process.
\end{remark}

\section*{Acknowledgements}

The authors are grateful to Arnaud Guillin and Laurent Miclo for their remarks
during the ANR EVOL meeting held in Hammamet (2010). They also thank the
anonymous referees for their helpful suggestions and comments. This
work was
partially supported by the French ANR Project EVOL.



%

\printhistory

\end{document}